\documentclass[a4paper]{article}

\usepackage{amssymb,amsmath,enumerate,color,ascmac,latexsym,diagbox}
\usepackage[legacycolonsymbols]{mathtools}
\usepackage{enumitem}
\usepackage{mathtools}
\usepackage{amsthm}
\usepackage{amsrefs}
\usepackage{autobreak}
\usepackage[geometry]{ifsym}
\usepackage[all]{xy}

\theoremstyle{plain}

\newtheorem{Thm}{Theorem}[section]
\newtheorem{df}[Thm]{Definition}
\newtheorem{Cor}[Thm]{Corollary}
\newtheorem{prop}[Thm]{Proposition}
\newtheorem{Lem}[Thm]{Lemma}

\newtheorem{Remark}{\textbf{Remark}}
\newtheorem{Main thm}{Main Theorem}

\newcommand{\five}{~~~~~}

\begin{document}
\title{\huge On alternating sum formulas for the Arakawa-Kaneko multiple zeta values}
\author{Yuta Nishimura}
\date{}
\maketitle

\begin{abstract}
In this paper, we construct generating functions of the Arakawa-Kaneko zeta values by the similar method to \cite{CE}. From the expressions, we show alternating sum formulas for them. Based on these results, we apply the same method to other zeta values.
\end{abstract}

\section{Introduction}
For integers $k_1,\ldots,k_{r-1}\geq1$ and $k_r\geq2$, the multiple zeta and zeta-star  value are defined respectively as follows:
\[
\zeta(k_1,\ldots,k_r)\coloneqq \sum_{0<m_1<\cdots<m_r}\frac{1}{m_1^{k_1}\cdots m_r^{k_r}},
\] 
\[
\zeta^{\star}(k_1,\ldots,k_r)\coloneqq \sum_{0<m_1\leq \cdots \leq m_r}\frac{1}{m_1^{k_1}\cdots m_r^{k_r}}.
\]
For an index $\textbf{k}=(k_1,k_2,\cdots,k_r)$, we define $|\textbf{k}|:=k_1+k_2+\cdots+k_r$ and dep$(\textbf{k}):=r$, called the weight and the depth, respectively. If $k_1,\ldots k_{r-1}\geq 1$ and $k_r\geq 2$, we call $\mathbf{k}$ admissible.\\
In this paper, we investigate the Arakawa-Kaneko multiple zeta function (which was introduced in \cite{AK}):
\begin{equation}
\xi(k_1,\ldots,k_r;s)\coloneqq \frac{1}{\Gamma(s)}\int^{\infty}_{0}\frac{t^{s-1}}{e^t-1}\mathrm{Li}_{k_1,\ldots,k_r}(1-e^{-t})dt, \nonumber
\end{equation}
where $r\geq1,~~k_1\ldots k_{r}\geq1,~~s\in \mathbb{C}~(\Re s> 0)$. Their special values at positive integers were studied in many papers. Our first main result concerns this value.
 
\begin{Main thm}[{Theorem \ref{main1}}]\
\begin{enumerate}[label=$\mathrm{(\arabic*)}$]
\item
For $k\geq 0$, we have
\[
\sum_{r+k_1+k_2=k}(-1)^{k_2}\xi(\{1\}^r,k_1+1;k_2+1)=
\left \{
\begin{array}{ll}
\zeta(k+2) &\five (k:\mathrm{even}),\\
0&\five(k:\mathrm{odd}).
\end{array}
\right.
\]

\item
For $r\geq1$ and $k\geq0$, we have
\begin{equation}
\begin{split}
\sum_{k_1+\cdots+ k_{r+1}=k}(-1)^{k_{r+1}}\xi&(k_1+1,\ldots,k_r+1;k_{r+1}+1)\\
&=
\sum_{a+2b=k}(-1)^a
\binom{a+r-1}{r-1}
\zeta^{\star}(a+r+1,\{2\}^b) .\nonumber
\end{split}
\end{equation}
\end{enumerate}
\end{Main thm}
Based on this result, we apply the same method to other multiple zeta values.

\begin{Main thm}[{Theorem \ref{main2}}]\
For $n,l\geq0$, we have
\begin{equation}
\begin{split}
\sum_{a+b=n}(-1)^b&\left(
\sum_{\substack{|\mathbf{k}|=l\\ \mathrm{dep(\mathbf{k})}=a+1}}\zeta^t(k_1+1,\ldots,k_a+1,k_{a+1}+b+2)\right)\\
&=
\sum_{a+2b=n}t^a\binom{a+l}{l}\zeta^{\star}(a+l+2,\{2\}^b),\nonumber
\end{split}
\end{equation}
where $\zeta^t$ is the interpolated multiple value $($introduced in \cite{Y1}$)$.
\end{Main thm}

\begin{Main thm}[Theorem\ref{main3}]
For $k,l\geq0$, we have
\begin{equation}
\begin{split}
\sum_{r+m=k}(-1)^m&
\left(
\sum_{k_1+\cdots+k_{r+1}=l}\xi_t(k_1+1,\ldots,k_{r+1}+1;m+1)
\right)\\
&=
\sum_{a+b=k}(-1)^at^b\binom{b+l}{b}\zeta(\{1\}^a,b+l+2),\nonumber
\end{split}
\end{equation}
where $\xi_t$ is the \textrm{t-Arakawa-Kaneko multiple zeta function} which we will introduce in later section.
\end{Main thm}

In this paper, $E_r$ denotes the area 
\[
\{(u_1,\ldots,u_r)\in[0,1]^r | 0<u_1<\cdots<u_r<1\},
\]
where $r$ is a positive integer greater than or equal to 1.

\section{Alternating sum for the Arakawa-Kaneko multiple zeta values}
In \cite{KO}, $\xi(\textbf{k},m)$ was expressed  by the following integral expression called Yamamoto's integral expression (introduced in \cite{Y2}).

\begin{Thm}[{\cite{KO}}]
\label{xi-yint}
For $k_1,\ldots, k_r\geq1$ and $m\geq1$, we have
\[
\xi(k_1,\ldots,k_r;m)=
I \left(\begin{xy}{
{(0,-22)\ar @{{*}-}(0,-18)},
{(0,-18)\ar @{o.}(0,-10)},
{(0,-10) \ar @{o-} (0,-6)},
{(0,-6) \ar @{{*}.} (0,2)},
{(0,2) \ar @{{*}-o} (0,6)}, 
{(0,6) \ar @{.o} (0,14)}, 
{(0,14) \ar @{-o} (5,18)}, 
{(10,2) \ar @{{*}-{*}} (10,6)}, 
{(10,6) \ar @{.{*}} (10,14)}, 
{(10,14) \ar @{-} (5,18)}, 
{(-1,2) \ar @/^1mm/ @{-} ^{k_r} (-1,14)}, 
{(11,2) \ar @/_1mm/ @{-} _{m-1} (11,14)}, 
{(-1,-22)\ar @/^1mm/@{-} ^{k_1} (-1,-10)},
}\end{xy} \right).
\]
\end{Thm}

From this, we get the iterated integral expression of the Arakawa-Kaneko multiple zeta values.

\begin{prop}\label{xi-int}\
\begin{enumerate} [label=$\mathrm{(\arabic*)}$]
\item For $k_1,k_2\geq0$ and $r\geq0$, we have
\begin{equation}
\begin{split}
\xi&(\{1\}^r,k_1+1;k_2+1)\\
&=\frac{1}{r!~k_1!~k_2!}\int_{E_2}\left(\log\frac{1}{1-u_1}\right)^r\left(\log\frac{u_2}{u_1}\right)^{k_1}\left(\log\frac{1}{1-u_2}\right)^{k_2}~\frac{du_1~du_2}{(1-u_1)u_2}.\nonumber
\end{split}
\end{equation}

\item For $k_1,\ldots,k_{r+1}\geq0$ and $r\geq1$, we have
\begin{equation}
\begin{split}
\xi&(k_1+1,\ldots,k_r+1;k_{r+1}+1)\\
=&\frac{1}{k_1!k_2!\cdots k_{r+1}!}\\
&\times\int_{E_{r+1}}\left(\log\frac{u_2}{u_1}\right)^{k_1}\cdots\left(\log\frac{u_{r+1}}{u_r}\right)^{k_r}\left(\log\frac{1}{1-u_{r+1}}\right)^{k_{r+1}}~\frac{du_1\cdots du_{r+1}}{(1-u_1)\cdots(1-u_r)u_{r+1}}.\nonumber
\end{split}
\end{equation}
\end{enumerate}
\end{prop}
\begin{proof}
(1) From Theorem \ref{xi-yint},
\begin{equation}\begin{split}
\xi(\{1\}^r,k_1+1,k_2+1)&=
I \left(\begin{xy}{
{(5,18)\ar@{o-}(0,14)},
{(0,-10)\ar @{{*}.}(0,-2)},
{(0,-2) \ar @{{*}-} (0,2)},
{(0,2) \ar @{{*}-o} (0,6)}, 
{(0,6) \ar @{.o} (0,14)}, 
{(10,2) \ar @{{*}-{*}} (10,6)}, 
{(10,6) \ar @{.{*}} (10,14)}, 
{(10,14) \ar @{-} (5,18)}, 
{(-1,6) \ar @/^1mm/ @{-} ^{k_1} (-1,14)}, 
{(11,2) \ar @/_1mm/ @{-} _{k_2} (11,14)}, 
{(-1,-10)\ar @/^1mm/@{-} ^{r} (-1,-2)},
}\end{xy} \right)\\
&=\frac{1}{r!k_1!k_2!}
\left(
\begin{xy}
{(0,10)\ar@{o-}(-12,-4)},
{(0,10)\ar@{-}(-4,-4)},
{(-12,-4)\ar@{.}(-4,-4)},
{(-12,-4)\ar@{o-}(-8,-8)},
{(-4,-4)\ar@{o-}(-8,-8)},
{(-8,-8)\ar@{{*}-{*}}(-12,-12)},
{(-8,-8)\ar@{-{*}}(-4,-12)},
{(-12,-12)\ar@{.}(-4,-12)},
{(0,10)\ar@{-{*}}(12,-4)},
{(0,10)\ar@{-{*}}(4,-4)},
{(4,-4)\ar@{.}(12,-4)},
{(-12,-13)\ar @/_1mm/@{}_{r}(-4,-13)},
{(-12,-3)\ar @/^1mm/@{-}^{k_1}(-4,-3)},
{(4,-5)\ar @/_1mm/@{-}_{k_2}(12,-5)},
\end{xy}
\right)\\
&=(r.h.s).\nonumber
\end{split}
\end{equation}
(2) can be proved by the same way as (1).
\end{proof}

Now we define generating functions of our alternating sums as follows.

\begin{df}\
\begin{enumerate}[label=$\mathrm{(\arabic*)}$]
\item For $k\geq0$, we define
\[
S_k\coloneqq \sum_{r+k_1+k_2=k}(-1)^{k_2}\xi(\{1\}^r,k_1+1;k_2+1).
\]
\item For $r\geq1$ and $k\geq0$, we define
\[
T_{r,k}=\sum_{k_1+\cdots+ k_{r+1}=k}(-1)^{k_{r+1}}\xi(k_1+1,\ldots,k_r+1;k_{r+1}+1).
\]
\item For $r\geq1$, we define
\begin{equation}
\begin{split}
F(x)&\coloneqq \sum_{k=1}^{\infty}S_kx^k,\\
G_r(x)&\coloneqq \sum_{k=1}^{\infty}T_{r,k}x^k,\nonumber
\end{split}
\end{equation}
where $|x|$ is sufficiently small.
\end{enumerate}
The convergence of $F(x)$ and $G_r(x)$ follow from the rough order estimations $S_k=O(3^k)~(k \rightarrow +\infty)$ and $T_{r,k}=O((r+2)^k)~(k \rightarrow +\infty)$.
\end{df}

By Proposition \ref{xi-int}, the functions $F(x)$ and $G_r(x)$ are immediately transformed into following integral expressions.

\begin{Lem}\label{xiG-int}\
\begin{equation}
\begin{split}
&F(x)=\int_{E_2}\left(\frac{1-u_2}{1-u_1}\right)^x\left(\frac{u_2}{u_1}\right)^x\frac{du_1du_2}{(1-u_1)u_2}.\\
&G_r(x)=\int_{E_{r+1}}\left(\frac{u_{r+1}}{u_1}\right)^x\left(1-u_{r+1}\right)^x\frac{du_1\cdots du_{r+1}}{(1-u_1)\cdots(1-u_r)u_{r+1}}.\nonumber
\end{split}
\end{equation}
\end{Lem}

We can get our first results from these integral expressions. To calculate them, the following Proposition (\cite{CE}) plays an important role.

\begin{prop}[{\cite{CE}}]\label{2}
For $l\geq1$, we have
\[
\frac{\Gamma(l-x)\Gamma(l+x)}{\Gamma(l)^2}=\prod_{k\geq l}\left(1-\frac{x^2}{k^2}\right)^{-1}.
\]
\end{prop}

At the end of this section, we prove our first alternating sum formulas.

\begin{Thm}\label{main1}\
\begin{enumerate}[label=$\mathrm{(\arabic*)}$]
\item
For $k\geq 0$, we have
\[
S_k=
\left \{
\begin{array}{ll}
\zeta(k+2) &\five (k:\mathrm{even}),\\
0&\five(k:\mathrm{odd}).
\end{array}
\right.
\]

\item
For $r\geq1$ and $k\geq0$, we have
\[T_{r,k}=
\sum_{a+2b=k}(-1)^a
\binom{a+r-1}{r-1}
\zeta^{\star}(a+r+1,\{2\}^b) .
\]
\end{enumerate}
\end{Thm}

\begin{proof}
(1) From Lemma \ref{xiG-int} and the generalized binomial theorem,
\begin{equation}
\begin{split}
F(x)&=
\sum_{l=1}^{\infty}\frac{\Gamma(l+x)}{\Gamma(l)\Gamma(1+x)}\int_{E_2}u_1^{l-x-1}(1-u_2)^x{u_2}^x\frac{du_1du_2}{u_2}\\
&=
\sum_{l=1}^{\infty}\frac{\Gamma(l+x)}{\Gamma(l)\Gamma(1+x)}\frac{1}{l-x}\int_{0}^{1}{u_2}^{l-1}(1-u_2)^xdu_2\\
&=
\sum_{l=1}^{\infty}\frac{\Gamma(l+x)}{\Gamma(l)\Gamma(1+x)}\frac{1}{l-x}B(l,x+1)\\
&=
\sum_{l=1}^{\infty}\frac{\Gamma(l+x)}{\Gamma(l)\Gamma(1+x)}\frac{1}{l-x}\frac{\Gamma(l)\Gamma(x+1)}{\Gamma(x+l+1)}\\
&=
\sum_{l=1}^{\infty}\left(l^2-x^2\right)^{-1}\\
&=
\sum_{n=1}^{\infty}\zeta(2(n+1))x^{2n}.\nonumber
\end{split}
\end{equation}
By comparing the coefficients on both sides, we get the conclusion.\\

(2) With the change of variables $(u_1,u_2,\ldots,u_{r+1})\rightarrow(1-u_{r+1},1-u_r,\ldots,1-u_1)$,
\[
G_r(x)=\int_{E_{r+1}}\left(\frac{1-u_1}{1-u_{r+1}}\right)^x{u_1}^x\frac{du_1\cdots du_{r+1}}{(1-u_1)u_2\cdots u_{r+1}}.
\]
By almost the same method as (1),
\begin{equation}
\begin{split}
G_r(x)&=
\sum_{l=1}^{\infty}\frac{\Gamma(l-x)}{\Gamma(l)\Gamma(1-x)}\int_{E_{r+1}}\left(1-u_{r+1}\right)^{-x}{u_1}^{l+x-1}\frac{du_1\cdots du_{r+1}}{u_2\cdots u_{r+1}}\\
&=
\sum_{l=1}^{\infty}\frac{\Gamma(l-x)}{\Gamma(l)\Gamma(1-x)}\frac{1}{l+x}\int_{E_{r}}\left(1-u_{r+1}\right)^{-x}{u_2}^{l+x-1}\frac{du_2\cdots du_{r+1}}{u_3\cdots u_{r+1}}\\
&=\cdots\\
&=
\sum_{l=1}^{\infty}\frac{\Gamma(l-x)}{\Gamma(l)\Gamma(1-x)}\frac{1}{(l+x)^{r}}\int^{1}_{0}\left(1-u_{r+1}\right)^{-x}{u_{r+1}}^{l+x-1}du_{r+1}\\
&=\sum_{l=1}^{\infty}\frac{\Gamma(l-x)}{\Gamma(l)\Gamma(1-x)}\frac{1}{(l+x)^{r}}B(1-x,l+x)\\
&=
\sum_{l=1}^{\infty}\frac{\Gamma(l-x)}{\Gamma(l)\Gamma(1-x)}\frac{1}{(l+x)^{r}}\frac{\Gamma(1-x)\Gamma(l+x)}{\Gamma(l+1)}\\
&=
\sum_{l=1}^{\infty}\frac{1}{l}\frac{\Gamma(l-x)\Gamma(l+x)}{\Gamma(l)^2}\frac{1}{(l+x)^r}.\nonumber
\end{split}
\end{equation}
Here, we use Proposition \ref{2} and get
\begin{equation}
\begin{split}
\sum_{l=1}^{\infty}\frac{1}{l}\frac{\Gamma(l-x)\Gamma(l+x)}{\Gamma(l)^2}&\frac{1}{(l+x)^r}\\
&=
\sum_{l=1}^{\infty}\frac{1}{l^{r+1}}\frac{\Gamma(l-x)\Gamma(l+x)}{\Gamma(l)^2}\left(\sum_{m=0}^{\infty}\frac{\Gamma(m+r)}{\Gamma(r)\Gamma(m+1)}\frac{(-x)^{m}}{l^m}\right)\\
&=\sum_{m=0}^{\infty}(-1)^m\binom{m+r-1}{r-1}\sum_{l=1}^{\infty}\frac{x^m}{l^{m+r+1}}\prod_{k\geq l}\left(1-\frac{x^2}{k^2}\right)^{-1}\\
&=\sum_{n=0}^{\infty}\left\{ \sum_{a+2b=n}(-1)^a\binom{a+r-1}{r-1}\cdot \zeta^{\star}(a+r+1,\{2\}^b)\right\}x^n. \nonumber
\end{split}
\end{equation}
By comparing the coefficients on both sides, we get the conclusion.
\end{proof}

\section{Alternating sums for the interpolated multiple zeta values}
In this section, we use a similar way in the previous section to get a generating function of ``height 1" t-MZVs. And from the expression, we compute our second alternating sum.

\begin{df}[{\cite{Y1}}]
For $\mathbf{k}=(k_1,\ldots,k_r):admissible$, we define
\begin{equation}
\zeta^t(k_1,\ldots,k_r)\coloneqq\sum_{\mathbf{k}'}t^{\sigma(\mathbf{k}')}\zeta(\mathbf{k}'),\nonumber
\end{equation}
where $\sigma(\mathbf{k}')\coloneqq r-{\mathrm{dep}}(\mathbf{k}')$, and the sum runs over all indices of the form $(k_1\Box k_2\Box\ldots \Box k_r)$  in which each  $\Box$ is filled by $``,"$ or $``+"$.
\end{df}

This function is called the ``interpolated multiple zeta value (t-MZV)''.
If we put $t=0$ (resp. $t=1$), this value coincides with the multiple zeta (resp.~zeta-star) value.

\begin{Remark}
In \cite{CE}, the generating function of $\mathrm{t}$-$\mathrm{MZVs}$ of fixed weight, depth and height has been constructed. In this paper, we treat a specialized case.
\end{Remark}

First, we prepare a Yamamoto's integral expression of height 1 t-MZVs.

\begin{prop}\label{tMZV-yint}
For $\mathbf{k}=(k_1,\ldots,k_r):admissible$, we have
\[
\zeta^t(\mathbf{k})
=I\left(
\begin{xy}
{(-28,-14)\ar@{{*}-}(-24,-12)},
{(-24,-12)\ar@{o.}(-16,-8)},
{(-16,-8)\ar@{o-}(-12,-6)},
{(-12,-6)\ar@{-}(-8,-4)},
(-12,-6)*{\odot},
{(-8,-4)\ar@{o.}(0,0)},
{(0,0)\ar@{o-}(4,2)},
(4,2)*{\odot},
{(4,2)\ar@{.}(12,6)},
(12,6)*{\odot},
{(12,6)\ar@{-}(16,8)},
{(16,8)\ar@{o.o}(24,12)},
{(-28,-15)\ar@/_2mm/@{-}_{k_1}(-16,-9)},
{(-12,-7)\ar@/_2mm/@{-}_{k_2}(0,-1)},
{(12,5)\ar@/_2mm/@{-}_{k_r}(24,11)}
\end{xy}
\right),
\]
where $\odot=(\frac{1}{1-u}+\frac{t}{u})du$.
\end{prop}
\begin{proof}
It is obvious from the definition. 
\end{proof}

From this, we get an iterated integral expression of height 1 t-MZVs and its generating function.

\begin{Lem}\label{tMZV-int}
For $p,q\geq0$ and $t\in \mathbb{R}$, we have
\[
\zeta^t(\{1\}^p,q+2)=\frac{1}{p!q!}\int_{E_2}\left(\log\frac{1-u_1}{1-u_2}+t\log\frac{u_2}{u_1}\right)^p\left(\log\frac{1}{u_2}\right)^{q}~\frac{du_1du_2}{(1-u_1)u_2}.
\]
\end{Lem}
\begin{proof}
From Proposition \ref{tMZV-yint},
\begin{equation}
\begin{split}
\zeta^t(\{1\}^p,q+2)&=I
\left(
\begin{xy}
{(0,-12)\ar@{{*}-}(0,-8)},
(0,-8)*{\odot},
{(0,-8)\ar@{.}(0,0)},
(0,0)*{\odot},
{(0,0)\ar@{-}(0,4)},
{(0,4)\ar@{o-}(0,8)},
{(0,8)\ar@{o.o}(0,16)},
{(1,-8)\ar@/_1mm/@{-}_p(1,0)},
{(1,8)\ar@/_1mm/@{-}_q(1,16)}
\end{xy}
\right)\\
&=
\frac{1}{p!q!}I\left(
\begin{xy}
(0,10)*{},
(0,-10)*{},
{(0,-10)\ar@{{*}-}(-6,-2)},
{(0,-10)\ar@{-}(6,-2)},
(-6,-2)*{\odot},
(6,-2)*{\odot},
{(6,-2)\ar@{-o}(0,2)},
{(-6,-2)\ar@{-}(0,2)},
{(0,2)\ar@{-o}(-6,6)},
{(0,2)\ar@{-o}(6,6)},
{(-6,-2)\ar@{.}(6,-2)},
{(-6,6)\ar@{.}(6,6)},
{(-6,-3)\ar@/_1mm/@{-}_{p}(6,-3)},
{(-6,7)\ar@/^1mm/@{-}^{q}(6,7)}
\end{xy}
\right)\\
&=
\frac{1}{p!q!}\int_{E_2}\left(\int_{u_1}^{u_2} \frac{1}{1-u}+\frac{t}{u}du\right)^p\left(\int_{u_2}^{1}\frac{1}{u}du\right)^{q}\cdot \frac{du_1du_2}{(1-u_1)u_2}\\
&=(r.h.s).\nonumber
\end{split}
\end{equation}
\end{proof}

\begin{Lem}\label{tMZVG-int}
For $t\in \mathbb{R}$, $|tx|<1$, $|y|<1$ and $|x+y|<1$, we have
\begin{equation}
\begin{split}
\sum_{p,q\geq0}\zeta^t(\{1\}^p,q+2)x^py^q&=
\int_{E_2}\left(\frac{1-u_1}{1-u_2}\right)^x\left(\frac{u_2}{u_1}\right)^{tx}\left(\frac{1}{u_2}\right)^y\frac{du_1du_2}{(1-u_1)u_2}\\
&=\sum_{l=1}^{\infty}\frac{\Gamma(l-x)\Gamma(l-y)}{\Gamma(l)\Gamma(l-x-y)}~\frac{1}{(l-x-y)(l-tx)}.\nonumber
\end{split}
\end{equation}
\end{Lem}
\begin{proof}
The first equality is immediately concluded from Lemma \ref{tMZV-int}.
The second equality is proved similar to Theorem \ref{main1}:
\begin{equation}
\begin{split}
U(x,y)&=\int_{E_2}\left(\frac{1-u_1}{1-u_2}\right)^x\left(\frac{u_2}{u_1}\right)^{tx}\left(\frac{1}{u_2}\right)^y\frac{du_1du_2}{(1-u_1)u_2}\\
&=\sum_{l=1}^{\infty}\frac{\Gamma(l-x)}{\Gamma(1-x)\Gamma(l)}\int_{E_2}(1-u_2)^{-x}u_1^{l-tx-1}u_2^{tx-y}\frac{du_1du_2}{u_2}\\
&=
\sum_{l=1}^{\infty}\frac{\Gamma(l-x)}{\Gamma(1-x)\Gamma(l)}\frac{1}{l-tx}\int_{E_2}(1-u_2)^{-x}u_2^{l-y-1}du_2\\
&=\sum_{l=1}^{\infty}\frac{\Gamma(l-x)}{\Gamma(1-x)\Gamma(l)}\frac{1}{l-tx}B(1-x,l-y)\\
&=\sum_{l=1}^{\infty}\frac{\Gamma(l-x)}{\Gamma(1-x)\Gamma(l)}\frac{1}{l-tx}\frac{\Gamma(1-x)\Gamma(l-y)}{\Gamma(l+1-y-x)}\\
&=
\sum_{l=1}^{\infty}\frac{\Gamma(l-x)\Gamma(l-y)}{\Gamma(l)\Gamma(l-x-y)}~\frac{1}{(l-x-y)(l-tx)}.\nonumber
\end{split}
\end{equation}
\end{proof}

\begin{Remark}
In \cite{ZC}, the left hand side of the above Lemma was calculated as 
$\frac{1}{(1-y)(1-xt)}
\;_3F_2\left({1-y,1+(1-t)x,1\atop 2-y,2-xt};1\right)$, 
where $\;_3F_2$ is a hypergeometric function.
On the other hand, we can write the right hand side of above result as \\$\frac{\Gamma(1-x)\Gamma(1-y)}{\Gamma(2-x-y)(1-xt)}
\;_3F_2\left({1-x,1-y,1-tx\atop 2-x-y,2-xt};1\right).$
Lemma \ref{tMZVG-int} can also be obtained by applying a transformation formula for $\;_3F_2$ to the result in \cite{ZC}.
\end{Remark}

In \cite{Y1}, Yamamoto discussed the case of differentiating t-MZVs with respect to $t$. By using it, various sum formulas are obtained from one sum formula.

\begin{prop}[{\cite{Y1}}]\label{tMZV-diff}
For $\textbf{k}:admissible$, we have
\[
\frac{d}{dt}\zeta^t(\mathbf{k})=\sum_{\sigma(\mathbf{k}')=1}\zeta^t(\mathbf{k}'),
\]
where the sum runs over all indices of the form $(k_1\Box k_2\Box\cdots \Box k_r)$  in which one of $\Box$ is filled by $``+"$ and the others are filled by $``,"$.
\end{prop}

If we put $x=-y$ in Proposition \ref{tMZVG-int}, the left hand side becomes the generating function of alternating sums for ``height 1'' t-MZVs and we get our second sum formula.

\begin{Thm}\label{main2}
For $n,l\geq0$ and $t\in \mathbb{R}$, we have
\begin{equation}
\begin{split}
\sum_{a+b=n}(-1)^b&\left(
\sum_{\substack{|\textbf{k}|=l\\ \textbf{dep}(\textbf{k})=a+1}}\zeta^t(k_1+1,\ldots,k_a+1,k_{a+1}+b+2)\right)\\
&=
\sum_{a+2b=n}t^a\binom{a+l}{l}\zeta^{\star}(a+l+2,\{2\}^b).\nonumber
\end{split}
\end{equation}
When $t=0$, we have
\begin{equation}
\begin{split}
\sum_{a+b=n}(-1)^b&\left(
\sum_{\substack{|\mathbf{k}|=l\\ \mathrm{dep}(\mathbf{k})=a+1}}\zeta(k_1+1,\ldots,k_a+1,k_{a+1}+b+2)\right)\\
&=\left\{
\begin{split}
&\zeta^{\star}(l+2,\{2\}^b)\five (n=2b)\\
&~~~~~~~0 ~~~~~~~~~~~~~(n=2b+1)
\end{split}
\right. .
\end{split}\nonumber
\end{equation}
When $l=0$, we have
\[
\sum_{a+b=n}(-1)^b\zeta^t(\{1\}^a,b+2)=\sum_{a+2b=n}t^a\zeta^{\star}(a+2,\{2\}^b).
\]
Therefore, this theorem contains an interpolation of the alternating sum formulas for height one multiple zeta and zeta-star values.
\end{Thm}
\begin{proof}
We begin with the case of $l=0$.
If we put $y=-x$ in Lemma \ref{tMZVG-int},
\[
\sum_{n=0}^{\infty}\left(\sum_{a+b=n}(-1)^b\zeta^t(\{1\}^a,b+2)\right)x^n=
\sum_{l=1}^{\infty}\frac{\Gamma(l-x)\Gamma(l+x)}{\Gamma(l)^2}\frac{1}{l(l-tx)}.
\]
By using Proposition \ref{2}, the right hand side is transformed 
\begin{equation}
\begin{split}
\sum_{l=1}^{\infty}\frac{\Gamma(l-x)\Gamma(l+x)}{\Gamma(l)^2}\frac{1}{l(l-tx)}&=\sum_{l=1}^{\infty}\frac{\Gamma(l-x)\Gamma(l+x)}{\Gamma(l)^2}\frac{1}{l^2}\sum_{m=0}^{\infty}\frac{(tx)^m}{l^m}\\
&=
\sum_{m=0}^{\infty}\sum_{l=1}^{\infty}\frac{(tx)^m}{l^{m+2}}\prod_{k\geq l}\left(1-\frac{x^2}{k^2}\right)^{-1}\\
&=
\sum_{n=0}^{\infty}\left(\sum_{a+2b=n}t^a\zeta^{\star}(a+2,\{2\}^b)\right)x^n. \nonumber
\end{split}
\end{equation}
By comparing the coefficients on both sides, the case of $l=0$ is proved.
The generalized case is proved by differentiating both sides of the above case $l$ times with respect $t$ and applying  Proposition \ref{tMZV-diff}. 
\end{proof}

\begin{Remark}
If we put $y=0$ in Lemma \ref{tMZVG-int} immediately, we get 
\[
\zeta^{t}(\{1\}^p,2)=(1+t+t^2+\cdots+t^p)\zeta(p+2).
\]
By differentiating both sides repeatedly with respect $t$, we get ordinary sum formulas for $\mathrm{t}$-$\mathrm{MZVs}$ $($\cite{Y1,ZC}$)$.
\end{Remark}

\section{Alternating sum for a certain t-Arakawa-Kaneko multiple zeta values}

In this section, we define a certain ``t-Arakawa-Kaneko multiple zeta value (t-AK-MZV)" which was introduced in \cite{OW2}, and discuss its alternating sums.

\begin{df}
For $r\geq1$, $k_1,\ldots,k_r\geq1$, and $m\geq1$, we define
\[
\xi_t(k_1,\cdots,k_r;m)\coloneqq \sum_{\mathbf{k}'}t^{\sigma(\mathbf{k}')}\xi(\mathbf{k}';m).
\]
where $\sigma(\mathbf{k}')\coloneqq r-\mathrm{dep}(\mathbf{k}')$, and the sum runs over all indices of the form $(k_1\Box k_2\Box\cdots \Box k_r)$  in which each  $\Box$ is filled by $``,"$ or $``+"$.
\end{df}

\begin{Remark}
In \cite{OW,OW2}, the interpolations between the Arakawa-Kaneko multiple zeta function and the Kaneko-Tsumura multiple zeta function:
\[
\xi^t (\mathbf{k};s)=\frac{1}{\Gamma(s)}\int^{\infty}_{0}x^{s-1}\frac{\mathrm{Li}^{t}_{\mathbf{k}}(1-e^{-x})}{e^x-1}dx
\],
\[
\eta^t (\mathbf{k};s)=\frac{1}{\Gamma(s)}\int^{\infty}_{0}x^{s-1}\frac{\mathrm{Li}^{t}_{\mathbf{k}}(1-e^{x})}{1-e^x}dx,
\]
where
\[
\mathrm{Li}^{t}_{\mathbf{k}}(z)=\sum_{\mathbf{k}\preceq \mathbf{k'}}t^{\mathbf{dep(k')}-r}\mathrm{Li}_{\mathbf{k'}}(z),
\]
was already constructed and several properties were studied. The above definition is different from them.
\end{Remark}

From the definition, we get a Yamamoto's integral expression of t-AK-MZV as follow.

\begin{prop}\label{tAK-MZV-yint}
For $r\geq1,~~k_1,\ldots,k_r\geq1,~~m\geq1$, we have
\[
\xi_t(k_1,\ldots,k_r;m)=
I\left(
\begin{xy}
(0,32)*{},
(-12,0)*{},
(12,0)*{},
{(0,28)\ar@{o-o}(-8,24)},
{(-8,24)\ar@{.o}(-8,16)},
{(-8,16)\ar@{-}(-8,12)},
(-8,12)*{\odot},
{(-8,12)\ar@{.}(-8,4)},
(-8,4)*{\odot},
{(-8,4)\ar@{-o}(-8,0)},
{(-8,0)\ar@{.o}(-8,-8)},
{(-8,-8)\ar@{-}(-8,-12)},
(-8,-12)*{\odot},
{(-8,-12)\ar@{-o}(-8,-16)},
{(-8,-16)\ar@{.o}(-8,-24)},
{(-8,-24)\ar@{-{*}}(-8,-28)},
{(0,28)\ar@{-{*}}(8,24)},
{(8,24)\ar@{.{*}}(8,16)},
{(-9,24)\ar@/_1mm/@{-}_{k_r}(-9,12)},
{(-9,0)\ar@/_1mm/@{-}_{k_2}(-9,-12)},
{(-9,-16)\ar@/_1mm/@{-}_{k_1}(-9,-28)},
{(9,24)\ar@/^1mm/@{-}^{m-1}(9,16)},
\end{xy}
\right),
\]
where $\odot=(\frac{1}{1-u}+\frac{t}{u})du$.
\end{prop}

\begin{proof}
It is proved easily from the definition and Theorem \ref{xi-yint}. 
\end{proof}
From this expression and definition, $\xi_0(k_1,\ldots,k_r;m)=\xi(k_1,\ldots,k_r;m)$ and $\xi_1(\{1\}^{r+1};m+1)=\eta(r+1;m+1)$, where $\eta(r+1;m+1)$ is the Kaneko-Tsumura multiple zeta function which is introduced in \cite{KT}

Next, we show analogies of  Lemma \ref{tMZV-int} and Proposition \ref{tMZV-diff} for height one t-AK-MZVs.

\begin{Lem}\label{tAK-MZV-int}
For $r\geq0,~~m\geq0$ and $t\in \mathbb{R}$, we have
\[
\xi_t(\{1\}^{r+1};m+1)=
\frac{1}{r!m!}\int_{E_2}
\left(
\log\frac{1-u_1}{1-u_2}+t\log\frac{u_2}{u_1}
\right)^r
\left(
\log\frac{1}{1-u_2}
\right)^{m}
\frac{du_1du_2}{(1-u_1)u_2}.
\]
\end{Lem}

\begin{proof}
From Proposition \ref{tAK-MZV-yint},
\begin{equation}
\begin{split}
\xi_t(\{1\}^{r+1};m+1)&=
I\left(
\begin{xy}
(0,12)*{},
(0,-12)*{},
{(0,8)\ar@{o-}(-8,4)},
{(-8,4)\ar@{.}(-8,-4)},
(-8,4)*{\odot},
{(-8,-4)\ar@{-{*}}(-8,-8)},
(-8,-4)*{\odot},
{(0,8)\ar@{-{*}}(8,4)},
{(8,4)\ar@{.{*}}(8,-4)},
{(-9,4)\ar@/_1mm/@{-}_{r}(-9,-4)},
{(9,4)\ar@/^1mm/@{-}^{m}(9,-4)},
\end{xy}
\right)\\
&=
\frac{1}{r!m!}
I\left(
\begin{xy}
(0,12)*{},
(0,-12)*{},
{(0,10)\ar@{o-}(-12,-4)},
{(0,10)\ar@{-}(-4,-4)},
{(-12,-4)\ar@{.}(-4,-4)},
{(-12,-4)\ar@{-}(-8,-8)},
(-12,-4)*{\odot},
{(-4,-4)\ar@{-{*}}(-8,-8)},
(-4,-4)*{\odot},
{(0,10)\ar@{-{*}}(12,-4)},
{(0,10)\ar@{-{*}}(4,-4)},
{(4,-4)\ar@{.}(12,-4)},
{(-12,-3)\ar @/^1mm/@{-}^{r}(-4,-3)},
{(4,-5)\ar @/_1mm/@{-}_{m}(12,-5)},
\end{xy}
\right)\\
&=(r.h.s).\nonumber
\end{split}
\end{equation}
\end{proof}

\begin{prop}\label{tAK-MZV-diff}
For $r\geq l\geq0,~m\geq0$ we have
\[
\frac{d^l}{dt^l}\xi_t(\{1\}^{r+1};m+1)=l!\sum_{k_1+\cdots+k_{r-l+1}=l}\xi_t(k_1+1,\ldots,k_{r-l+1}+1;m+1).
\]
\end{prop}

\begin{proof}
By Proposition \ref{tAK-MZV-yint},
\begin{equation}
\begin{split}
\frac{d^l}{dt^l}&\xi_t(\{1\}^{r+1};m+1)\\
&=\frac{1}{(r-l)!m!}
\int_{E_2}
\left(
\log\frac{1-u_1}{1-u_2}+t\log\frac{u_2}{u_1}
\right)^{r-l}
\left(
\log\frac{u_2}{u_1}
\right)^l
\left(
\log\frac{1}{1-u_2}
\right)^{m}
\frac{du_1du_2}{(1-u_1)u_2}\\
&=l!
I\left(
\begin{xy}
(0,-12)*{},
(0,12)*{},
{(0,8)\ar@{o-}(-12,4)},
{(0,8)\ar@{-o}(-4,4)},
(-12,4)*{\odot},
{(-12,4)\ar@{.}(-12,-4)},
{(-12,-4)\ar@{-{*}}(-8,-8)},
(-12,-4)*{\odot},
{(-4,4)\ar@{.}(-4,-4)},
{(-4,-4)\ar@{o-o}(-8,-8)},
{(0,8)\ar@{-{*}}(6,4)},
{(6,4)\ar@{.{*}}(6,-4)},
{(-13,4)\ar@/_1mm/@{-}_{r-l}(-13,-4)},
{(-3,4)\ar@/^1mm/@{-}^{l}(-3,-4)},
{(7,4)\ar@/^1mm/@{-}^{m}(7,-4)},
\end{xy}
\right)\\
&=(r.h.s).\nonumber
\end{split}
\end{equation}
\end{proof}

Our third alternating sum formulas are obtained by calculating following generating function.

\begin{Lem}\label{tAK-MZVG-int}
For $t\in\mathbb{R}$, $|tx|<1,~~|x|<1,~~|x+y|<1$, we have
\begin{equation}
\begin{split}
\sum_{k,m=0}^{\infty}\xi_t(\{1\}^{r+1};m+1)
x^ky^m
&=
\sum_{l=1}^{\infty}\frac{\Gamma(l-x)\Gamma(1-x-y)}{\Gamma(1-x)\Gamma(l+1-x-y)}\frac{1}{l-tx}\nonumber
\end{split}
\end{equation}
\end{Lem}

\begin{proof}
From Lemma \ref{tAK-MZV-int},
\begin{equation}
\begin{split}
\sum_{k,n=0}^{\infty}\xi_t(\{1\}^{r+1};m+1)
x^ky^n
&=
\int_{E_2}
\left(
\frac{1-u_1}{1-u_2}
\right)^x
\left(
\frac{u_2}{u_1}
\right)^{tx}
\left(
\frac{1}{1-u_2}
\right)^y
\frac{du_1du_2}{(1-u_1)u_2}\\
&=
\sum_{l=1}^{\infty}
\frac{\Gamma(l-x)}{\Gamma(l)\Gamma(1-x)}
\int_{E_2}
{u_1}^{l-tx-1}
\left(
\frac{1}{1-u_2}
\right)^{x+y}
u_2^{tx-1}
du_1du_2\\
&=
\sum_{l=1}^{\infty}
\frac{\Gamma(l-x)}{\Gamma(l)\Gamma(1-x)}
\frac{1}{l-tx}
\int_{0}^{1}
(1-u_2)^{-x-y}
u_2^{l-1}
du_2\\
&=
\sum_{l=1}^{\infty}
\frac{\Gamma(l-x)}{\Gamma(l)\Gamma(1-x)}
\frac{1}{l-tx}
B(l,1-x-y)\\
&=
\sum_{l=1}^{\infty}
\frac{\Gamma(l-x)}{\Gamma(l)\Gamma(1-x)}
\frac{1}{l-tx}
\frac{\Gamma(l)\Gamma(1-x-y)}{\Gamma(l+1-x-y)}\\
&=
\sum_{l=1}^{\infty}\frac{\Gamma(l-x)\Gamma(1-x-y)}{\Gamma(1-x)\Gamma(l+1-x-y)}\frac{1}{l-tx}.\nonumber
\end{split}
\end{equation}
\end{proof}

If we put $y=-x$ in Lemma \ref{tAK-MZVG-int}, an alternating sum formula for "height one" t-AK-MZVs is obtained.

\begin{prop}\label{sum-int}
For $k\geq0$ and $t\in \mathbb{R}$, we have
\[
\sum_{r+m=k}(-1)^m\xi_t(\{1\}^{r+1};m+1)=\sum_{a+b=k}(-1)^at^b\zeta(\{1\}^a,b+2).
\]
In particular, when $t=0$, 
\[
\sum_{r+m=k}(-1)^m\xi(\{1\}^{r+1};m+1)=(-1)^k\zeta(k+2).
\]
\end{prop}
\begin{proof}
If we put $x=-y$ in Lemma \ref{tAK-MZVG-int},
\begin{equation}
\begin{split}
\sum_{k=0}^{\infty}
\Big(
\sum_{r+m=k}&(-1)^m\xi_t(\{1\}^{r+1};m+1)
\Big)x^k\\
&=
\sum_{l=1}^{\infty}
\frac{\Gamma(l-x)}{\Gamma(l)\Gamma(1-x)}
\frac{1}{l-tx}
\frac{\Gamma(1)\Gamma(l)}{\Gamma(l+1)}\\
&=
\sum_{l=1}^{\infty}
\frac{\Gamma(l-x)}{\Gamma(l)\Gamma(1-x)}
\frac{1}{l(l-tx)}\\
&=
\sum_{l=1}^{\infty}
\left(
1-x
\right)
\left(
1-\frac{x}{2}
\right)
\cdots
\left(
1-\frac{x}{l-1}
\right)
\frac{1}{l(l-tx)}\\
&=
\sum_{n=0}^{\infty}
\sum_{l=1}^{\infty}
\left(
1-x
\right)
\left(
1-\frac{x}{2}
\right)
\cdots
\left(
1-\frac{x}{l-1}
\right)
\frac{x^nt^n}{l^{n+2}}\\
&=
\sum_{k=0}^{\infty}
\left(
\sum_{a+b=k}(-1)^at^b\zeta(\{1\}^a,b+2)
\right)x^k.\nonumber
\end{split}
\end{equation}
By comparing the coefficients on both sides, we get the conclusion.
\end{proof}

The above result can be regarded as an interpolation between ``height one" alternating sum for $\xi$ and ``depth two" alternating sum for $\eta$.
In fact, by putting $t=1$ in this formula, we get the following sum formula which was conjectured in \cite{KT} and proved by Kaneko-Sakata.

\begin{Cor}
For $k\geq0$, we have
\begin{equation}
\begin{split}
\sum_{r+m=k}(-1)^m\eta(r+1;m+1)
&=\sum_{a+b=k}(-1)^a\zeta(\{1\}^a,b+2)\\
&=\left\{
\begin{split}
&2\left(1-\frac{1}{2^{2c+1}}\right)\zeta(2c+2)\five (k=2c),\\
&~~~~~~~~~~~~~0 ~~~~~~~~~~~~~~~~~~~~(k=2c+1).
\end{split}\nonumber
\right.
\end{split}
\end{equation}
\end{Cor}

\begin{proof}
The first equality is immediately concluded by putting $t=1$ in Proposition \ref{sum-int}. The second equality is the very result in \cite{Le}.
\end{proof}

Our third alternating sum formula can be obtained by differentiating both sides of Proposition  \ref{sum-int} repeatedly with respect $t$.

\begin{Thm}\label{main3}
For $k,l\geq0$ and $t\in \mathbb{R}$ we have
\begin{equation}
\begin{split}
\sum_{r+m=k}(-1)^m&
\left(
\sum_{k_1+\cdots+k_{r+1}=l}\xi_t(k_1+1,\ldots,k_{r+1}+1;m+1)
\right)\\
&=
\sum_{a+b=k}(-1)^at^b\binom{b+l}{b}\zeta(\{1\}^a,b+l+2).\nonumber
\end{split}
\end{equation}
In particular, when $t=0$,
\[
\sum_{r+m=k}(-1)^m
\left(
\sum_{k_1+\ldots+k_{r+1}=l}\xi(k_1+1,\ldots,k_{r+1}+1;m+1)
\right)
=
(-1)^k\zeta(\{1\}^k,l+2).
\]
\end{Thm}

\begin{proof}
This is proved by differentiating both sides of Theorem \ref{sum-int} $l$ times with respect $t$ and applying  Proposition \ref{tAK-MZV-diff}. 
\end{proof}

\section*{Acknowledgment}
The author would like to thank my supervisor Professor Hirofumi Tsumura for his kind advice and helpful comments. The author also thanks Kyosuke Nishibiro, a senior of the lab, for his advice on writing paper. Finally the author also thanks Professor Shuji Yamamoto and Professor Yasuo Ohno for their valuable suggestions.

\end{document}